\pdfoutput=1

\documentclass[reqno]{amsart}
\usepackage[foot]{amsaddr}
\usepackage{amssymb}
\usepackage{eucal}
\usepackage[hyperindex,breaklinks]{hyperref}
\usepackage{bm}
\usepackage{dsfont}
\usepackage{enumitem}
\setlist[enumerate]{label = \textup{(\alph*)}}

\usepackage[T2A,OT1]{fontenc}
\usepackage{csquotes}
\usepackage[russian,english]{babel}
\usepackage[style = ext-numeric, autolang=other]{biblatex}
\makeatletter
\newcommand{\subalign}[1]{%
  \vcenter{%
    \Let@ \restore@math@cr \default@tag
    \baselineskip\fontdimen10 \scriptfont\tw@
    \advance\baselineskip\fontdimen12 \scriptfont\tw@
    \lineskip\thr@@\fontdimen8 \scriptfont\thr@@
    \lineskiplimit\lineskip
    \ialign{\hfil$\m@th\scriptstyle##$&$\m@th\scriptstyle{}##$\hfil\crcr
      #1\crcr
    }%
  }%
}
\makeatother

\theoremstyle{definition}\newtheorem{Def}{Definition}
\theoremstyle{plain}\newtheorem{Th}{Theorem}
\theoremstyle{plain}\newtheorem*{ThNoNum}{Theorem}

\theoremstyle{remark}
\theoremstyle{plain}\newtheorem{Le}{Lemma}
\theoremstyle{plain}
\theoremstyle{plain}\newtheorem{Cor}{Corollary}
\theoremstyle{plain}\newtheorem{Prop}{Proposition}

\newcommand{\Cii}[1]{_{{}_{\scriptstyle #1}}}
\newcommand{\cii}[1]{_{{}_{#1}}}
\newcommand{\Av}[2]{\langle {#1}\rangle_{{}_{\scriptstyle #2}}}
\newcommand{\av}[2]{\langle {#1}\rangle_{{}_{#2}}}
\newcommand{\E}{\mathbb{E}}
\newcommand{\cnde}[2]{\E\,[{#1}\mid{#2}]}
\newcommand{\df}{\buildrel{}_\mathrm{def}\over=}
\newcommand{\chr}{\mathds{1}}
\newcommand{\Rpls}{\mathbb{R}_{\scriptscriptstyle\ge 0}}
\newcommand{\Zpls}{\mathbb{Z}_{\scriptscriptstyle\ge 0}}
\newcommand{\ui}{[0,1)}
\newcommand{\Haar}{\bm{\mathrm{H}}}

\newcommand{\seq}[1]{l_{\scriptscriptstyle\!#1}^2}

\newcommand{\ind}{\mathcal{A}}
\newcommand{\bell}{\bm B}
\newcommand{\btype}{\mathcal{B}}
\DeclareMathOperator{\osc}{osc}
\DeclareMathOperator{\supp}{supp}
\DeclareMathOperator{\clsp}{\overline{span}}
\DeclareMathOperator{\spn}{span}
\DeclareMathOperator{\sign}{sign}

\hypersetup{pdfstartview = FitH}

\title[Bellman function for operators on martingales]{Bellman function method for general operators on martingales}

\author[Viacheslav Borovitskiy]{Viacheslav Borovitskiy$^{\ast,\dagger}$}
\author[Nikolay N. Osipov]{Nikolay N. Osipov$^{\ast,\mathsection}$}
\author[Anton Tselishchev]{Anton Tselishchev$^{\ast,\dagger}$}

\address{$^\ast$St. Petersburg Department of V.~A.~Steklov Institute of Mathematics of the Russian Academy of Sciences, St. Petersburg, Russia}
\address{$^\dagger$St. Petersburg State University, Chebyshev Laboratory, St. Petersburg, Russia}
\address{$^\mathsection$HSE University, International Laboratory of Game Theory and Decision Making, St. Petersburg, Russia}
\email{nicknick AT pdmi DOT ras DOT ru}

\thanks{The work is supported by the Foundation for the Advancement of Theoretical Physics and Mathematics ``BASIS''. The second author also acknowledges the support of the HSE University Basic Research Program. The first and third authors acknowledge the support of the Ministry of Science and Higher Education of the Russian Federation (agreement No. 075-15-2019-1620).}


\begin{document}
\begin{abstract}
    It is shown that the Bellman function method can be applied to study the $L^p$-norms of general operators on martingales, i.e., of operators that are not necessarily martingale transforms. Informally, we provide a single Bellman-type function that ``encodes'' the $L^p$-boundedness of ``almost all'' operators from Gundy's extrapolation theorem. As examples of such operators, we consider the Haar transforms and the operator whose $L^p$-boundedness underlies Rubio de Francia's inequality for the Walsh system.
\end{abstract}
\maketitle
\setcounter{secnumdepth}{1}
\setcounter{tocdepth}{1}

\section{Introduction}

Gundy's theorem~\cite{Gu1968} of 1968 can be seen as a martingale analog of the principle that general Calder\'on--Zygmund operators are bounded.
Consider, for example, Rubio de Francia's inequality~\cite{Ru1985}. As discussed in the introduction of~\cite{MaOs2019}, it can be thought of as a one-sided analog of Parseval's identity in $L^p$. This inequality can be reduced to 
the $L^p$-boundedness of a certain Cal\-de\-r\'on--Zyg\-mund operator. 
On the other hand, its analog~\cite{Os2016tran} for the Walsh basis, the martingale counterpart of the Fourier basis, 
follows from the version~\cite{Ki1985tran} of Gundy's theorem for operators on vector-valued martingales. 
This illustrates the generality of Gundy's theorem because the mentioned Cal\-de\-r\'on--Zyg\-mund operator arising in Rubio de Francia's original considerations~\cite{Ru1985} has a vector-valued kernel that satisfies only a very weak and subtle smoothness condition.


On the other hand, in the paper~\cite{Bu1984} of 1984, Burkholder proves the $L^p$-bo\-un\-ded\-ness of the martingale transforms
using another approach borrowed from stochastic optimal control. The martingale transforms are a special case of operators from Gundy's theorem. They can be considered as a martingale analog of the Hilbert transform, the most basic example of a Calder\'on--Zygmund operator. 
Nevertheless, Burkholder's work~\cite{Bu1984} is now regarded as a real breakthrough in harmonic analysis because of his method: it gives a deep insight into the structure of the estimated $L^p$-norms and, in particular, allows him to calculate sharp constants in the corresponding $L^p$-inequalities.
His approach gives rise to a new theory \cite{SIZOV2015,Ose2012,VaVo2020}, which establishes a deep connection between harmonic analysis, stochastic processes, differential geometry, and partial differential equations. The methodology suggested by Burkholder
is now commonly referred to as the Bellman function method in harmonic analysis.

In this paper, we implement the program that was partly announced (without any proofs) in 
the short report~\cite{BoOsTs2021tran}: we show that Burkholder's method can be extended to ``almost all'' operators from the vector Gundy's theorem~\cite{Ki1985tran} and provide a single Bell\-man-ty\-pe function that ``encodes'' 
their $L^p$-boundedness. In particular, this applies to 
the martingale Rubio de Francia operator 
from~\cite{Os2016tran}.

\section{Notation and preliminaries}
Let $\ind$ be at most a countable set of indices. 
By $\seq{\ind}$ we denote the corresponding $l^2$~space where elements 
of sequences are indexed by $\alpha\in\ind$. Further, $L^p$ and $L^p(\seq{\ind})$ denote the Lebesgue spaces, respectively, of scalar-valued and vector-valued functions on the interval $\ui$.

Suppose $n$ runs over $\Zpls$ and $\{\mathcal{F}_n\}$ is a filtration of the Borel algebra over $\ui$. 
We put
$
\mathcal{F}_\infty \df \sigma(\cup_n\mathcal{F}_n).
$
For $f \in L^1(\seq{\ind})$, we denote 
$$\E_m f\df \cnde{f}{\mathcal{F}_m},\quad m \in \Zpls \cup \{\infty\}.$$ 
Concerning the operators~$\E_m$ for vec\-tor-va\-lu\-ed functions and the properties of vec\-tor-va\-lu\-ed martingales discussed below, see, e.g.,~\cite[Chapter V]{DiUh1977}.
A~sequence~$\{f_n\}$ of $\seq{\ind}$-valued functions $$f_n = \{f_n^\alpha\}\Cii{\alpha \in \ind} \in L^1(\seq{\ind})$$ is called a martingale 
if 
$
\E_m f_n = f_{m}
$
for $m\le n$. 
The $L^p$-norm of a martingale is defined as 
$$
\|\{f_n\}\|\Cii{L^p} \df \sup_n \|f_n\|\Cii{L^p},
$$
and any function $f \in L^p(\seq{\ind})$, $1\le p <\infty$, generates a martingale $\{\E_n f\}$ such that 
$$
\E_n f \xrightarrow{L^p} \E_\infty f\quad\mbox{and}\quad   \|\{\E_n f\}\|\Cii{L^p} =\|\E_\infty f\|\Cii{L^p}.
$$

We call a martingale simple if $f_{n+1}\equiv f_n$ for all sufficiently large~$n$.
We also impose on $\{\mathcal{F}_n\}$ the regularity condition~\cite[condition~(R)]{Ki1985tran}:  
$\mathcal{F}_n$ are finite and the measures of their atoms 
decrease, as $n$ increases, no faster than a geometric progression. 
We present a version of Gundy's theorem for vector-valued martingales that is formulated and proved in~\cite[Theorem~1]{Ki1985tran} in somewhat greater generality. The original scalar theorem can be found in~\cite{Gu1968}.
\begin{ThNoNum}[R.~F.~Gundy]
Let $T$ be an operator that transforms simple martingales $f = \{f_n\}$ into scalar-valued measurable functions and has the following properties.
\begin{enumerate}[label = \textup{(G\arabic*)}]
    \item\label{it:G1} $|T(f+g)| \le C_1(|Tf| + |Tg|)$.
    \item $\|Tf\|\Cii{L^2} \le C_2\|f\|\Cii{L^2}$.
    \item\label{it:G3} If $f$ satisfies the relations $\Delta_0f \df f_0\equiv \bm 0$ and 
    $$
        \Delta_n f \df f_n-f_{n-1} = \chr_{e_n} \Delta_n f
    $$
    for $n>0$ and some $e_n\in\mathcal{F}_{n-1}$, then 
    $$
    \{|Tf| > 0\} \subset \bigcup_{n>0} e_n.
    $$
\end{enumerate}
    For such an operator, we have
    $$
    \bigl|\{|Tf| > \lambda\}\bigr| \le C\bigl(C_1,C_2,\{\mathcal{F}_n\}\bigr)\,\lambda^{-1}\|f\|\Cii{L^1}\quad\mbox{for}\quad \lambda >0.
    $$
\end{ThNoNum}

By ``$\sqsubseteq$'' we denote the relation ``is a dyadic subinterval of''. Further, we consider, along with $\ui$, 
an arbitrary bounded interval 
$I \sqsubseteq \mathbb{R}$ and deal only with filtrations constructed from its dyadic subintervals. For any subinterval $J\sqsubseteq I$, we denote its left and right halves by $J^\pm$ and introduce 
the $\seq{\ind}$-valued functions $\chr^\alpha_J$, $\alpha \in \ind$: the component of $\chr^\alpha_J$ with index~$\alpha$ is the indicator function~$\chr_J$ and the other components are zero. 
Then the $\seq{\ind}$-valued Haar functions
\begin{equation}\label{eq:basis}
h^{\alpha}_0 \df |I|^{-1/2}\chr^{\alpha}_{I}\quad\mbox{and}\quad h^{\alpha}_J \df |J|^{-1/2}(\chr^{\alpha}_{J^+} - \chr^{\alpha}_{J^-}),
\quad J\sqsubseteq I,\;\alpha \in \ind,
\end{equation}
form an orthonormal basis in~$L^2\bigl(I,\seq{\ind}\bigr)$. We drop the index~$\alpha$ in~\eqref{eq:basis} in situations where we are in the scalar setting. 

By $\mathcal{L}\bigl(I,\seq{\ind}\bigr)$ we denote the set of all linear operators that transform finite linear combinations of the vector-valued Haar functions~\eqref{eq:basis} into measurable scalar functions.
Next, we introduce a subset of $\mathcal{L}\bigl(I,\seq{\ind}\bigr)$ that is a somewhat more regular analog of the class of operators from Gundy's theorem.
\begin{Def}\label{def:class_G}
We say that an operator $T\in\mathcal{L}\bigl(I,\seq{\ind}\bigr)$ 
belongs to the class $\mathcal{G}\bigl(I,\seq{\ind}\bigr)$ if it has the following properties.
\begin{enumerate}[label = \textup{(R\arabic*)}]
    \item\label{it:R1} $\|Tf\|\Cii{L^2} \le \|f\|\Cii{L^2}$. 
    \item\label{it:R2} The operator~$T$ does not enlarge the supports of the basis functions: 
    $$
        \supp Th_J^\alpha \subseteq J \quad\mbox{for $\alpha \in \ind$ and $J\sqsubseteq I$}.
    $$
\end{enumerate}
\end{Def}
Let $\{\mathcal{F}_n\}$ be the Haar filtration on $\ui$ where the atomic intervals are bisected one by one, from left to right: 
\begin{align*}
\mathcal{F}_0 &= \sigma\{\ui\},\quad \mathcal{F}_1 = \sigma\big\{\big[0,\tfrac{1}{2}\big),\big[\tfrac{1}{2},1\big)\big\},\quad
\mathcal{F}_2 = \sigma\big\{\big[0,\tfrac{1}{4}\big), \big[\tfrac{1}{4},\tfrac{1}{2}\big), \big[\tfrac{1}{2},1\big)\big\},\\[3pt] 
\mathcal{F}_3 &= \sigma\big\{\big[0,\tfrac{1}{4}\big), \big[\tfrac{1}{4},\tfrac{1}{2}\big), \big[\tfrac{1}{2},\tfrac{3}{4}\big), \big[\tfrac{3}{4}, 1\big)\big\},\quad \mathcal{F}_4 = \sigma\big\{\big[0,\tfrac{1}{8}\big),\big[\tfrac{1}{8},\tfrac{1}{4}\big),\big[\tfrac{1}{4},\tfrac{1}{2}\big),\dots\big\},
\quad\dots
\end{align*}
For this specific filtration, we prove that the only additional regularity, assumed by Definition~\ref{def:class_G} 
as compared to \ref{it:G1}--\ref{it:G3}, is that $T$ is linear and contractive.
\begin{Prop}\label{pr:G_is_for_Gundy} 
Consider functions $f\in\spn\big(\{h^\alpha_0,h^\alpha_J\}\cii{{\alpha \in \ind, J \sqsubseteq \ui}}\big).$ For an operator $T\in\mathcal{L}\bigl(\ui,\seq{\ind}\bigr)$, condition~\ref{it:R2} is equivalent to condition~\ref{it:G3} for martingales $\{f_n\} \df \{\E_n f\}$, 
where $\E_n f$ are calculated with respect to the Haar filtration introduced above.
\end{Prop}
\begin{proof}
Suppose $T$ satisfies~\ref{it:R2}. Consider $f$ such that 
$$
\Delta_0 f \df \E_0 f = \sum_{\alpha\in\ind}(f, h_0^{\alpha})\,h_0^{\alpha} \equiv \bm 0.
$$
We have 
$$
\Delta_n f \df \E_nf - \E_{n-1}f
= \sum_{\alpha \in \ind} (f, h_{J_n}^{\alpha})\,h_{J_n}^{\alpha},
$$
where $J_n$ is the interval that is bisected when switching from $\mathcal{F}_{n-1}$ to $\mathcal{F}_n$.
Thus, the set $e_n = \{\Delta_n f \ne 0\}$ is empty if $(f, h_{J_n}^{\alpha}) = 0$ for all $\alpha\in\ind$, or equals~$J_n$ otherwise.
Due to~\ref{it:R2}, the following relation holds, implying~\ref{it:G3}:
$$
Tf = \sum_n\chr_{e_n}\sum_{\alpha \in \ind} (f, h_{J_n}^{\alpha})\,Th_{J_n}^{\alpha}.
$$

The reverse implication is obvious: we only need to apply~\ref{it:G3} to~$h_{J}^{\alpha}$.
\end{proof}

We can treat $T\in \mathcal{G}\bigl(I,\seq{\ind}\bigr)$ 
as bounded linear operators from $L^2\bigl(I,\seq{\ind}\bigr)$ to $L^2(I)$. We can also apply Gundy's theorem to obtain the we\-ak-ty\-pe~$(1,1)$ estimate for them. Thus, relying on the Marcinkiewicz interpolation theorem, we can prove 
that these operators are uniformly bounded in any $L^p$, $1<p\le 2$. 
On the other hand, the generality of the class $\mathcal{G}\bigl(I,\seq{\ind}\bigr)$
is, to a large extent, close to the generality of conditions \ref{it:G1}--\ref{it:G3}.

For sequences $\varepsilon = \{\varepsilon_J\}\cii{J\sqsubseteq \ui}$ of numbers $\varepsilon_J \in \{-1,1\}$, consider the operators $T_{\varepsilon}\in\mathcal{G}(\ui,\mathbb{R})$ defined by the formula 
\begin{equation*}
T_{\varepsilon}f \df (f,h_0)\,h_0 + \!\!\sum_{J\sqsubseteq \ui}\!\varepsilon_J\, (f,h_J)\, h_J. 
\end{equation*}
Such operators are called martingale transforms. Relying on the Bellman function method borrowed from stochastic optimal control, Burkholder suggested an alternative proof~\cite{Bu1984} of their $L^p$-boundedness that gives a deep insight into the structure of the norms $\|T_\varepsilon\|\cii{L^p\to L^p}$.
In particular, it allowed him to calculate 
$\sup_{\varepsilon} \|T_{\varepsilon}\|\cii{L^p\to L^p}$. 

Our goal is to show that the Bellman function method can be extended 
to the whole class $\mathcal{G}\bigl(I,\seq{\ind}\bigr)$. 
Our considerations are somewhat similar to ones in~\cite{NaTr1996tran} where a Bellman-type function is built for Burkholder's problem.
But since we consider much more general problem, everything becomes substantially more complicated.
We provide only one of the variety of appropriate Bellman-type functions, but this is enough to demonstrate the applicability of Burkholder's theory.

\section{Motivating examples}
First, we provide two examples of operators in~$\mathcal{G}$ that are not martingale transforms: exact Bellman function from~\cite{Bu1984} or Bellman-type function from~\cite{NaTr1996tran} cannot be used to prove their $L^p$-boundedness. Throughout this section, we assume that $I = \ui$ and that $\{\mathcal{F}_n\}$ is the standard dyadic filtration: 
\begin{equation*}
\mathcal{F}_{n}
=
\sigma
\Bigl\{
    \bigl[\tfrac{k}{2^n}, \tfrac{k+1}{2^n}\bigr)
    \,\Bigm|\,
    0 \leq k < 2^n
\Bigr\}
,\quad n\in\Zpls.
\end{equation*}
Here, in contrast to the Haar filtration, all the atomic intervals in $\mathcal{F}_{n-1}$ are bisected at once (not one by one) when switching to $\mathcal{F}_n$.

\subsection{Haar transforms.}
For each $n\in\Zpls$, we can build an 
operator ${\Haar_n\in\mathcal{L}(\ui,\mathbb{R})}$ 
that establishes a one-to-one correspondence between the finite Haar basis and the indicator basis in the subspace of 
$\mathcal{F}_{n}$-me\-as\-ur\-ab\-le functions, 
i.e. between
\begin{equation*}
    \bigl\{h_0, h_J\bigr\}_{J \sqsubseteq \ui,\,|J|\ge 2^{-n+1}}
    \qquad
    \mbox{and}
    \qquad
    \bigl\{|e|^{-1/2}\chr_e\bigr\}_{e \sqsubseteq \ui,\, |e|= 2^{-n}},
\end{equation*}
while mapping other Haar functions to zero or to themselves and satisfying~\ref{it:R2}.
Indeed, put $\Haar_0 h_0 \df h_0$. 
Next, suppose we have built~$\Haar_n$. 
For $|J|\ge 2^{-n+1}$ and for~$0$ substituted for $J$, we set
$$\Haar_{n+1}h_J \df |e^+|^{-1/2}\chr_{e^+},$$
where $e$ is such that
$\Haar_{n}h_J = |e|^{-1/2}\chr_{e}$. For $|J| = 2^{-n}$, we set
$$\Haar_{n+1}h_J \df |J^-|^{-1/2}\chr_{J^-}.$$ 
Each operator $\Haar_n$ acts on an $\mathcal{F}_{n}$-me\-as\-ur\-ab\-le function 
as the corresponding \emph{Haar transform} and 
gives a step function constructed from its Haar coefficients. 

\begin{samepage}
Furthermore, $\Haar_n$ is either a contraction or a unitary operator, depending on how we handle the Haar functions~$h_J$ with $|J| < 2^{-n+1}$, and thus it satisfies~\ref{it:R1}. 
Hence, we have the following fact.
\begin{Prop}
The operators~$\Haar_n$ belong to $\mathcal{G}(\ui,\mathbb{R})$.
\end{Prop}
\end{samepage}

\subsection{The Rubio de Francia operator.}
Consider the Walsh basis $\{w_n\}$ consisting of all possible products of 
Rademacher functions:
\begin{itemize}
    \item set $w_0 \df \chr_{\ui}$; 
    \item for any index $n>0$, consider its dyadic decomposition 
    $n = 2^{k_1} + \cdots + 2^{k_s}$, $k_1 >k_2>\cdots> k_s \ge 0$, and set
    $$
	    w_n(x) \df \prod_{i=1}^s r_{k_i+1}(x),
    $$
    where $r_k(x) = \sign\sin 2^k\pi x$.
\end{itemize}
The system~$\{w_n\}$ resembles in its properties the Fourier basis 
(see, e.g.,~\cite[\S 4.5]{KaSa1999tran}) and can be considered as its discrete analog.\footnote{It actually \emph{is} the Fourier basis in the sense of abstract harmonic analysis if we identify $\ui$ with the Cantor dyadic group in the right way.}
In particular, the following direct analog of Rubio de Francia's 
inequality~\cite{Ru1985} for the Walsh system is proved in~\cite{Os2016tran}.
\begin{ThNoNum}
Let $\{f^n\}$ be at most a countable collection of functions with Walsh spectra supported in pairwise disjoint intervals $I_n \subset \Zpls$\textup{:} 
$$
    f^n = \sum_{m\in I_n} (f^n,w_m)\,w_m.
$$
For $1<p\le 2$, we have
\begin{equation}\label{eq:Walsh_Rubio}
\Big\|\sum_n f^n\Big\|_{L^p} \le C_p\,\big\|\{f^n\}\big\|_{L^p(l^2)},
\end{equation}
where $C_p$ does not depend on $\{f^n\}$ or $\{I_n\}$.
\end{ThNoNum}
Our second example is the operator whose $L^p$-bo\-un\-ded\-ness underlies
inequality~\eqref{eq:Walsh_Rubio}. In order to introduce it, we need certain 
simple and well-known properties of the Walsh functions. 
\begin{enumerate}[label = \textup{(W\arabic*)}]
    \item\label{it:W1}  
    For a function $f\in L^1$, its martingale differences $\Delta_n f$ 
    with respect to $\{\mathcal{F}_n\}$ coincide with the Walsh multipliers associated with the indicator functions of the intervals ${D_0 \df \{0\}}$ and ${D_n = \{2^{n-1}, 2^{n-1}+1,\dots,2^n-1\}}$, $n>0$:
    \begin{equation}\label{eq:W1}
    \begin{aligned}
        \Delta_0 f &= (f,h_0)\,h_0 = (f,w_0)\,w_0,\\
        \Delta_n f &= 
        \hspace{-10pt}\sum_{\substack{J \sqsubseteq \ui\\ |J|= 2^{-n+1}}}\hspace{-8pt} (f,h_J)\,h_J
        =\!\!\sum_{m \in D_n}\! (f,w_m)\,w_m.
    \end{aligned}
    \end{equation}
    \item\label{it:W2} For $a,b\in \Zpls$, we have the ``exponential'' property 
    $w_a\, w_b \equiv w_{a\dotplus b}$, where 
    $a\dotplus b$ means the bitwise XOR operation: 
    the corresponding bits in the binary decompositions of $a$ and $b$ are summed modulo $2$.
\end{enumerate}
Suppose multi-indices $(j,k)$ run over a subset $\ind \subseteq \Zpls^2$ and 
numbers $a_{j,k} \in \Zpls$ are such that the sets $a_{j,k} \dotplus D_k$ 
are pairwise disjoint. 
Consider functions $$f=\bigl\{f^{j,k}\bigr\}_{(j,k)\in\ind}\in L^2\bigl(\seq{\ind}\bigr).$$
We introduce an operator~$G$ that
transplants parts of the Walsh spectra of $f^{j,k}$ into $a_{j,k} \dotplus D_k$ and combines the results into a single function:
$$
    Gf\df \!\!\!\sum_{(j,k)\in\mathcal{A}} \!\!w_{a_{j,k}}\,\Delta_k f^{j,k}.
$$
The paper~\cite{Os2016tran} mainly consists of combinatorial arguments that reduce 
estimate~\eqref{eq:Walsh_Rubio} to the estimate
$$
\|Gf\|\Cii{L^p} \le C_p\,\|f\|\Cii{L^p(\seq{\ind})},\quad 1<p\le 2,
$$
where $C_p$ depends only on~$p$.
The operator~$G$ satisfies the conditions of Gundy's theorem for the standard dyadic filtration. But we can easily switch to the Haar filtration. Namely, we have the following proposition.
\begin{Prop}
The operator~$G$ belongs to $\mathcal{G}\big(\ui,\seq{\ind}\big)$.
\end{Prop}
\begin{proof}
Since $a_{j,k} \dotplus D_k$ 
are pairwise disjoint, Parseval's identity for the Walsh basis and properties~\ref{it:W1} and~\ref{it:W2} imply~\ref{it:R1}:
$$
\|Gf\|\Cii{L^2}^2 = 
\!\!\!\sum_{(j,k)\in\mathcal{A}} \!\!\big\|\Delta_k f^{j,k}\big\|_{L^2}^2
\le \!\!\!\sum_{(j,k)\in\mathcal{A}} \!\!\big\|f^{j,k}\big\|_{L^2}^2=
\|f\|\Cii{L^2(\seq{\ind})}^2.
$$

Since we can express $\Delta_k f^{j,k}$ in terms of $h_0$ and $h_J$ as in~\eqref{eq:W1},
we have
\begin{itemize}
    \item if $k>0$ and $|J| = 2^{-k+1}$, then
    $
    Gh_J^{j,k} = w_{a_{j,k}}h_J
    $;
    \item $G$ vanishes on all other $h_J^{j,k}$.
\end{itemize}
This implies~\ref{it:R2}.
\end{proof}

\section{Main results}
Henceforth, we suppose $I\sqsubseteq\mathbb{R}$, $1<p\le 2$, and $\tfrac{1}{p}+\tfrac{1}{q} = 1$.
For $\varphi\in L^1(I)$, we denote $\Av{\varphi}{I} \df \frac{1}{|I|}\int_I\varphi$. We agree that for vectors $x,y \in \seq{\ind}$, the notation~$xy$ means their inner product, and $|x|$ means the $l^2$-norm of $x$. 
For $$f = \{f^\alpha\}\Cii{\alpha \in \ind} \in L^2\bigl(I, \seq{\ind}\bigr),$$ 
we set
$$
    \Av{f}{I} \df \bigl\{\Av{f^\alpha}{I}\bigr\}_{\alpha \in \ind} \quad\mbox{and}\quad 
    \osc_I^2(f) \df \Bigl\langle\bigl|f-\Av{f}{I}\bigr|^2\Bigr\rangle_{I}
    =\bigl\langle|f|^2\bigr\rangle_{I} - \bigl|\Av{f}{I}\bigr|^2.
$$
Suppose $T\in\mathcal{G}\big(I,\seq{\ind}\big)$. 
We have $\|T^*\|\cii{L^2\to L^2} =\|T\|\cii{L^2\to L^2}
\le 1$, and thus the inequality
\begin{equation}\label{eq:x2_ineq}
    \Av{g^2}{I} - \osc_I^2(T^*g) \ge \bigl|\Av{T^*g}{I}\bigr|^2
\end{equation}
holds for any $g\in L^2(I)$. It becomes an equality if $T^*$ is an isometry.
% ...

We introduce the Bellman function
\begin{equation}\label{eq:bell_def}
	\bell(x) \df \sup\left\{\bigl\langle g\,T\bigl[f-\Av{f}{I}\bigr]\bigr\rangle\Cii{I}
	\,\middle|\; 
	\begin{aligned}
	&\Av{f}{I} = x_1,\;
	\Av{g^2}{I} - \osc_I^2(T^*g) = x_2,\\[1pt] 
	&\Av{|f|^p}{I} = x_3,\; \Av{|g|^q}{I} = x_4
	\end{aligned}
	\right\},
\end{equation}
where $x = (x_1,x_2,x_3,x_4)\in \seq{\ind}\times\Rpls^3$ and the supremum is taken over $f\in L^2\big(I, \seq{\ind}\big)$, $g\in L^2(I)$, and $T\in\mathcal{G}\big(I,\seq{\ind}\big)$ satisfying the identities after the vertical bar. It is easy to check that $\bell$ does not depend on the choice of~$I$. 

Let $\Omega_{\bell}$ consist of all the points~$x$ for which the supremum in~\eqref{eq:bell_def} is taken over a non-empty set. 
Applying Jensen's 
inequality in vector and scalar forms (or H\"older's inequality together with Minkowski's integral inequality for the $\seq{\ind}$-norm), we obtain
$$
    \Omega_{\bell} \subseteq \Omega_p \df \bigl\{x \in \seq{\ind}\times\Rpls^3\bigm| |x_1|^p \le x_3,\, x_2 \le x_4^{2/q} \bigr\}.
$$ 
We introduce the class $\mathcal{K}^p\big(\seq{\ind}\big)$ of Bellman-type functions.
\begin{Def}\label{def:class_K}
We say that a function $B\in C(\Omega_p)$ belongs to the class $\mathcal{K}^p\big(\seq{\ind}\big)$ if it satisfies the following boundary condition and geometric concave-type condition.
\begin{enumerate}[label = \textup{(B\arabic*)}]
    \item\label{it:bnd_cnd} If $|x_1|^p = x_3$ then $B(x) \ge 0$.
    \item\label{it:cnc_cnd} If for $x, x^{\pm} \in \Omega_p$ and $\Delta \in \mathbb{R}$ we have
    \begin{equation}\label{eq:x_mean}
        \frac{x^+ + x^-}{2} - x = (\bm{0}, \Delta^2, 0, 0),
    \end{equation}
    then
    \begin{equation}\label{eq:main_ineq}
        B(x) \ge \frac{|x_1^+-x_1^-|}{2}\,|\Delta| + \frac{B(x^+)+B(x^-)}{2}.
    \end{equation}
\end{enumerate}
\end{Def}
Our first theorem states that any Bellman-type function majorizes the true Bellman function. 
\begin{Th}\label{th:BleB}
If $B\in\mathcal{K}^p\big(\seq{\ind}\big)$, then $\bell(x) \le B(x)$ for all $x\in \Omega_{\bell}$.
\end{Th}
Next, we provide a specific representative of the class~$\mathcal{K}^p\big(\seq{\ind}\big)$. For $y \in \Rpls^4$, we define the function
\begin{multline}\label{eq:B_0}
    \btype_0(y) \df 2(y_3 + y_4)-y_1^p-y_2^{q/2}\\
    -\delta_p
    \begin{cases}
        y_1^{2-p}y_2 + y_1^{2-p-2t_p(p-1)}y_2^{t_p+1}, & y_1^p \ge y_2^{q/2};\\[5pt]
        \frac{2}{q}(2+t_p)\,y_2^{q/2} + \frac{2}{p}(2-p-t_p(p-1))\,y_1^p, & y_1^p \le y_2^{q/2}.
    \end{cases}
\end{multline}
\begin{Th}\label{th:BinK}
    There exist parameters $t_p\ge 0$, $\delta_p > 0$, and a constant $C_p>0$ such that the restriction of the function 
    \begin{equation}\label{eq:B}
        \btype(x) \df C_p \btype_0(|x_1|, x_2, x_3, x_4)
    \end{equation}
    to $\Omega_p$ 
    belongs to $\mathcal{K}^p\big(\seq{\ind}\big)$.
\end{Th}
Theorems~\ref{th:BleB} and~\ref{th:BinK} have the following consequence.
\begin{Cor}\label{cor:Lp_bnd}
If $T\in\mathcal{G}\big(\ui,\seq{\ind}\big)$, then for $f \in L^2\big(\seq{\ind}\big)$, we have
$$
    \|Tf\|\Cii{L^p} \le (2p^{1/p}q^{1/q}C_p + 1)\, \|f\|\Cii{L^p}.
$$
\end{Cor}
Finally, our third theorem concerns properties~\ref{it:bnd_cnd} and~\ref{it:cnc_cnd} for the Bellman function itself.
\begin{Th}\label{th:bell_is_cnc}
Let $\bell_0$ be the function that is defined by~\eqref{eq:bell_def} in the situation where $\seq{\ind} = \mathbb{R}$ and $T$ runs only over unitary 
operators in $\mathcal{G}\big(I,\mathbb{R}\big)$. Then $\Omega_{\bell_0} = \Omega_p$ and $\bell_0$ satisfies \ref{it:bnd_cnd} and~\ref{it:cnc_cnd}.
\end{Th}
Theorems~\ref{th:BleB} and~\ref{th:bell_is_cnc} lead to the following conjectures (which may help to calculate the true Bellman function~$\bell$):
\begin{itemize}
    \item the supremum in~\eqref{eq:bell_def} is attained on unitary operators;
    \item the function $\bell$ is the pointwise infimum of the functions from $\mathcal{K}^p\big(\seq{\ind}\big)$.
\end{itemize}

\section{Guessing the candidate.}
In this section, we briefly describe how we guess~$\btype_0$. 
First, applying Taylor expansion to~\eqref{eq:main_ineq} and differentiating with respect to~$\Delta$, we infer 
that property~\ref{it:cnc_cnd} is associated with 
the differential inequality
\begin{equation}\label{eq:diff_form_guess}
    d^2{B}[y](dy) \le \frac{(dy_1)^2}{2\,\partial_{y_2}{B}(y)} \le 0.
\end{equation}
On the left, we calculate the Hessian at $y \in \Rpls^4$ and apply it, 
as a quadratic form, to an arbitrary vector $dy \in \mathbb{R}^4$. 
Using the ideas of~\cite{NaTr1996tran}, 
we come to our first guess: 
$$B(y)=2(y_3+y_4)-y_1^p-y_2^{q/2}.$$ For this function, condition~\eqref{eq:diff_form_guess}
takes the form 
$$
-p(p-1)y_1^{p-2}(dy_1)^2-\tfrac{q}{2}\big( \tfrac{q}{2}-1
\big)y_2^{{q}/{2}-2}(dy_2)^2\leq
-\frac{(dy_1)^2}{qy_2^{{q}/{2}-1}}.
$$
Therefore, we see that~$B$
satisfies~\eqref{eq:diff_form_guess} only for $y$ such that $y_1^{p-2}y_2^{{q}/{2}-1}\geq 1$ or, what is the same, where $y_2^{{q}/{2}}\geq y_1^p$. 
In order to fix this, we can try to add a correction term that makes $B$ ``more concave'', similarly to how it is done in~\cite{NaTr1996tran}. 
We refer to $\{y_2^{{q}/{2}} = y_1^p\}$ as the critical curve.
We try to add $-\delta_p\, y_1^{2-p}y_2$ below the critical curve and to add $$-\delta_p\Big( \tfrac{2}{q}y_2^{q/2}+\tfrac{2-p}{p}
y_1^p \Big)$$ above the critical curve (the latter expression comes from Young's
inequality). However, a direct computation (which is quite long) shows that
the resulting function satisfies~\eqref{eq:diff_form_guess} only for $q\leq 4 \Leftrightarrow
p\geq 4/3$! In order to overcome this restriction, we add the term $-\delta_p
y_1^{2-p-2t_p(p-1)}y_2^{t_p+1}$ below the critical curve (and the corresponding term
above). And this is how we come up with the Bellman
candidate.

\section{Proof of Theorem~\ref{th:BleB}}
Fix $x\in\Omega_{\bell}$ and consider $f\in L^2\big(I,\seq{\ind}\big)$, $g\in L^2(I)$, and $T\in\mathcal{G}\big(I,\seq{\ind}\big)$ such that $x=x^I$, where
$$
x^J = \bigl(x_1^J,x_2^J,x_3^J,x_4^J\bigr) \df \big(\Av{f}{J},\,\Av{g^2}{J} - \osc_J^2(T^*g),\,\Av{|f|^p}{J},\,\Av{|g|^q}{J}\big),\quad J\sqsubseteq I.
$$
We also introduce
$$
\delta_J \df \big\{|J|^{-1/2}(g,Th_J^{\alpha})\big\}_{\alpha\in\ind} 
\quad\mbox{and}\quad\Delta_J \df |\delta_J|.
$$
We have 
\begin{equation}\label{eq:osc_ser}
    \osc_J^2(T^*g) = \frac{1}{|J|}\sum_{Q\sqsubseteq J,\,\alpha\in\mathcal{A}} \big(g,Th_Q^{\alpha}\big)^2
\end{equation}
and $$\Delta_J^2 = \frac{x_2^{J^+}+x_2^{J^-}}{2}-x_2^J.$$
We note that property~\ref{it:R2} and equation~\eqref{eq:osc_ser} imply that the inequality $x_2^I \ge 0$ (see~\eqref{eq:x2_ineq}) is inherited by all $x_2^J$, $J\sqsubseteq I$, and this is the very place where we need~\ref{it:R2}.

We also obtain
$$
\frac{|x_1^{J^+} - x_1^{J^-}|}{2}\Delta_J \ge \frac{x_1^{J^+} - x_1^{J^-}}{2}\cdot\delta_J = \frac{1}{|J|}\sum_{\alpha \in \ind} (f,h_J^\alpha)(g,Th_J^\alpha).
$$
Applying inequality~\eqref{eq:main_ineq} $k$ times, we obtain
\begin{equation}\label{eq:ind_ineq_1}
    B(x) \ge \frac{1}{|I|}\sum_{\substack{J\sqsubseteq I,\,\alpha\in\ind\\|J|\ge 2^{-k+1}}} (f,h_J^\alpha)(g,Th_J^\alpha)
    +\hspace{-10pt}\sum_{\substack{J\sqsubseteq I\\|J|=2^{-k}|I|}}\hspace{-10pt}\frac{B(x^J)}{2^k}.
\end{equation}
We denote the first and second terms in~\eqref{eq:ind_ineq_1} by $U_k$ and $V_k$, respectively. Since the operator~$T$ and the inner product are 
continuous in $L^2$, we have 
$$
U_k \to \Av{g\,TPf}{I},
$$
where $P$ is the orthogonal projection onto $\clsp \big(\{h^\alpha_J\}\cii{{\alpha \in \ind, J \sqsubseteq I}}\big)$.

By $L^\infty_{00}\big(I,\seq{\ind}\big)$ we denote the subspace of $L^\infty\big(I,\seq{\ind}\big)$ consisting of bounded vector functions with a finite number of non-zero components. 
Assume for a while that $g\in L^\infty(I)$ and $f \in L^\infty_{00}\big(I,\seq{\ind}\big)$. 
We introduce the step function
\[
x^k(t) = (x_1^k(t),x_2^k(t),x_3^k(t),x_4^k(t))
\]
that takes values $x^J$ on the intervals $J \sqsubseteq I$, $|J|=2^{-k}$. We note that the functions $x_2^k(t)$ form a bounded submartingale. Thus, by the Lebesgue differentiation theorem and by Doob's martingale convergence theorem, we have
$$
    x^k \xrightarrow{\mathrm{a.e.}} (f, \eta, |f|^p, |g|^q),
$$
where $\eta$ is a function in $L^1(I)$.
All $x^k$ are uniformly bounded vector functions whose non-zero components take values in a finite-dimensional subspace of 
$\seq{\ind}\times\Rpls^3$. These values form a compact set in this subspace and, therefore, 
the continuous function~$B$ is bounded on the values of~$x^k$. 
Lebesgue's dominated convergence theorem and the boundary condition~\ref{it:bnd_cnd} imply that
$$
    V_k = \int\limits_I B(x^k(t))\,dt \to \int\limits_I B\bigl(f(t), \eta(t), |f(t)|^p, |g(t)|^q\bigr)\,dt \ge 0.
$$
We come to the inequality
\begin{equation*}
B(x) \ge \Av{g\,TPf}{I}.
\end{equation*}

Now we drop the assumptions $f \in L^\infty_{00}\big(I,\seq{\ind}\big)$ and $g\in L^\infty(I)$ and, instead, consider sequences $f_n$ and $g_m$ in 
these spaces that tend to $f$ and $g$ in the $L^2$- and $L^q$-norms, respectively. 
We have
$$
    B\bigl(\Av{f_n}{I},\,
    \Av{g_m^2}{I} - \osc_I^2(T^*g_m),\,
    \Av{|f_n|^p}{I},\,\Av{|g_m|^q}{I}\bigr) \ge \Av{g_m\,TPf_n}{I}.
$$
Relying on the continuity of $T$, $T^*$, $P$, and the inner product in $L^2$, we can pass to the limit as $n,m\to\infty$.\qed 

\section{Proof of Theorem~\ref{th:BinK}}
\begin{Le}\label{le:tech_le}
Suppose $y_1,y_2>0$ and $y_1^p \ge y_2^{q/2}$. For any $a_1,a_2,b_1,b_2 \in \mathbb{R}$ such that 
$$q(a_2-a_1) = 2p(b_1-b_2) \quad\mbox{and}\quad a_2- a_1 \ge 0,$$
we have $y_1^{a_1}y_2^{b_1}\le y_1^{a_2}y_2^{b_2}$.
\end{Le}
\begin{proof}
Raising both parts of $y_1^p \ge y_2^{q/2}$ to the power $\frac{a_2-a_1}{p}$, we obtain the desired inequality.
\end{proof}

\begin{Le}\label{le:B_0} 
There exist parameters $t_p\ge 0$, $\delta_p> 0$, and a constant $c_p>0$ such that 
the function~$\btype_0$, defined by~\eqref{eq:B_0}, belongs to $C^1(\Rpls^4)$ and we have
    $$
        d^2\btype_0[y](dy) \le c_p\frac{(dy_1)^2}{2\,\partial_{y_2}\btype_0(y)} \le 0
    $$
    for any vector $dy \in \mathbb{R}^4$ and any point $y\in \Rpls^4$ where $y_1,y_2\ne 0$ and $y_1^p \ne y_2^{q/2}$.
\end{Le}
\begin{proof}
The direct calculation of $\partial_{y_1}\btype_0$ and $\partial_{y_2}\btype_0$
implies that $\btype_0 \in C^1(\Rpls^4)$ and that ${\partial_{y_2}\btype_0 \le 0}$ for any $t_p, \delta_p\ge 0$. 
It remains to prove that
$$
\begin{pmatrix}
\partial_{y_1}^2\btype_0 - \frac{c_p}{2\,\partial_{y_2}\btype_0} & \partial_{y_1}\partial_{y_2}\btype_0\\[5pt]
\partial_{y_1}\partial_{y_2}\btype_0 & \partial_{y_2}^2\btype_0
\end{pmatrix}
\le 0,
$$
provided $y_1,y_2\ne 0$ and $y_1^p \ne y_2^{q/2}$. By direct calculations, we have $\partial_{y_2}^2\btype_0 \le 0$ for any $t_p, \delta_p\ge 0$.
Thus, it suffices to choose $t_p$, $\delta_p$, and $c_p$ such that
\begin{equation}\label{eq:minor_1}
    2\,\partial_{y_2}\btype_0\,\partial_{y_1}^2\btype_0 - c_p \ge c_p,
\end{equation}
and then to prove that
\begin{equation}\label{eq:minor_2}
    2\,\partial_{y_2}\btype_0\,(\partial_{y_1}\partial_{y_2}\btype_0)^2 - c_p\,\partial_{y_2}^2\btype_0\ge 0.
\end{equation}

First, suppose $y_1^p \le y_2^{q/2}$. Then inequality~\eqref{eq:minor_1} takes the form
$$
y_1^{p-2}y_2^{q/2-1}\bigl[p(p-1) + 2\delta_p(2-p-t_p(p-1))(p-1)\bigr]\bigl[q/2+\delta_p(2+t_p)\bigr]\ge c_p.
$$
In the case being considered, we have $y_1^{p-2}y_2^{q/2-1} = \big(y_1^{-p}y_2^{q/2}\big)^{1-2/q}\ge 1$ and $\partial_{y_1}\partial_{y_2}\btype_0 = 0$.
Therefore, inequalities~\eqref{eq:minor_1} and~\eqref{eq:minor_2} hold, for example, when $t_p\le\frac{2-p}{p-1}$, $\delta_p\ge 0$, 
and $c_p\le {1}/{2}$.

Next, suppose $y_1^p \ge y_2^{q/2}$. 
Fact~\ref{le:tech_le} implies that 
\begin{equation}\label{eq:y1y2}
\begin{aligned}
&\max\Bigl(y_1^{-p}y_2^{q/2},\;\; 
y_1^{s_p-2}y_2^{t_p+q/2}\Bigr)\le y_1^{p-2}y_2^{q/2-1},\\
&\max\Bigl(y_1^{s_p-p}y_2^{1+t_p},\;\;  
y_1^{2-2p}y_2,\;\;
y_1^{2s_p-2}y_2^{1+2t_p}\Bigr) \le y_1^{s_p+p-2}y_2^{t_p},
\end{aligned}
\end{equation}
where $s_p \df 2-p-2t_p(p-1)$, $t_p\le \frac{2-p}{2(p-1)}=\frac{q}{2}-1$, and $0 \le t_p \le 1$. Direct calculations give
\begin{align*}
\partial_{y_2}\btype_0(y)
&=
-\Bigl(\tfrac{q}{2}\, y_2^{q/2-1}
 +\delta_p\, y_1^{2-p}
 +(1+t_p) \delta_p\, y_1^{s_p} y_2^{t_p}\Bigr),\\
\partial_{y_1}^2\btype_0(y) &= -\Bigl(p(p-1)\,y_1^{p-2}
 +(2-p)(1-p)\delta_p\, y_1^{-p} y_2
 +s_p(s_p-1) \delta_p\, y_1^{s_p-2} y_2^{1+t_p}\Bigr).
\end{align*}
Using direct computation and estimates~\eqref{eq:y1y2}, we obtain
$$
    \partial_{y_2}\btype_0(y)\partial_{y_1}^2\btype_0(y) \ge 
    \delta_p p(p-1) + \tfrac{q}{2}\beta_1\, y_1^{p-2}y_2^{q/2-1} + \delta_p\beta_2\, y_1^{s_p+p-2}y_2^{t_p},
$$
where
\begin{align*}
\beta_1 &= p (p-1)
 - \delta_p (2-p) (p-1)
 - \delta_p s_p (1-s_p),
\\ 
\beta_2 &=  p (p-1) (1+t_p)
 - \delta_p (2-p) (p-1) (2+t_p)
 - \delta_p s_p (1-s_p) (2+t_p).
\end{align*}
By putting $c_p = \delta_p p(p-1)$ and by taking a sufficiently small~$\delta_p$, we come to~\eqref{eq:minor_1}.

It remains to prove~\eqref{eq:minor_2}. Fact~\ref{le:tech_le} implies 
$y_1^{s_p}y_2^{t_p} \le y_1^{2-p}$ and $y_1^{s_p-1}y_2^{t_p} \le y_1^{1-p}$. Relying on these inequalities, 
we obtain
\begin{align*}
    \partial_{y_2}\btype_0(y) &\ge -\bigl(\tfrac{q}{2}\,y_2^{q/2-1} +\delta_p(2+t_p)\,y_1^{2-p}\bigr),\\
    \partial_{y_1}\partial_{y_2}\btype_0 &= 
    -\delta_p\Bigl((2-p)\, y_1^{1-p} + s_p(1+t_p)\, y_1^{s_p-1}y_2^{t_p}\Bigr)\\
    &\ge -\delta_p\bigl(2-p +s_p(1+t_p)\bigr)\,y_1^{1-p}.
\end{align*}
Therefore, 
the expression in \eqref{eq:minor_2} is greater than
\begin{equation}\label{eq:minor_2_less}
c_p\tfrac{q}{2}\big(\tfrac{q}{2}-1\big)\, y_2^{q/2-2} + c_p\delta_p(1+t_p)t_p\, y_1^{s_p}y_2^{t_p-1} 
-\delta_p^2\beta_3\, y_1^{2-2p}y_2^{q/2-1} - \delta_p^3\beta_4\,y_1^{4-3p},
\end{equation}
where
$$
\beta_3 = q\big(2-p+s_p(1+t_p)\big)^2\quad\mbox{and}\quad \beta_4 = 2(2+t_p)\big(2-p+s_p(1+t_p)\big)^2.
$$
Again, Fact~\ref{le:tech_le} implies $y_1^{2-2p} y_2^{q/2-1} \le y_2^{q/2-2}$ and $y_1^{4-3p} \le y_1^{s_p}y_2^{t_p-1}$.
Therefore, expression~\eqref{eq:minor_2_less} is greater than
$$
    \bigl(c_p\tfrac{q}{2}\big(\tfrac{q}{2}-1\big) - \delta_p^2\beta_3\bigr) \, y_2^{q/2-2} + 
    \bigl(c_p\delta_p(1+t_p)t_p - \delta_p^3\beta_4\bigr)\, y_1^{s_p}y_2^{t_p-1}.
$$
Taking $c_p=\delta_pp(p-1)$, $t_p = \min\big(1,\frac{q}{2}-1\big)$, and a sufficiently small $\delta_p>0$, we finish the proof.
\end{proof}


\begin{Le}\label{le:B}
Suppose $\mathcal{A}$ is finite. There exist parameters $t_p\ge 0$, $\delta_p> 0$, and a constant $C_p >0$ such that 
the function~$\btype$, defined by~\eqref{eq:B}, belongs to $C^1(\seq{\ind}\times\Rpls^3)$ and we have
    \begin{equation}\label{eq:diff_form}
        d^2\btype[x](dx) \le \frac{|dx_1|^2}{2\,\partial_{x_2}\btype(x)} \le 0
    \end{equation}
    for any vector $dx \in\seq{\ind}\times\mathbb{R}^3$ and any point $x\in \seq{\ind}\times\Rpls^3$ where $|x_1|,x_2\ne 0$ and $|x_1|^p \ne x_2^{q/2}$.
\end{Le}
\begin{proof}
Let $x_1 = \{x_1^{\alpha}\}\Cii{\alpha\in\ind}$ and $r\df (|x_1|, x_2, x_3, x_4) \in \Rpls^4$.
First, we note that for any $\alpha \in \ind$, the function
$$
    \partial_{x_1^\alpha}\btype(x) = C_p\,\partial_{y_1}\btype_0(r)\frac{x_1^\alpha}{|x_1|}
$$
is continuous where $x_1 \ne \bm{0}$. We also have 
$$
\bigl|\partial_{x_1^\alpha}\btype(x)\bigr| 
\le C_p\bigl|\partial_{y_1}\btype_0(r)\bigr| \to 0
\quad\mbox{as}\quad x_1 \to \boldsymbol{0}.
$$
Thus, $\btype \in C^1(\seq{\ind}\times\Rpls^3)$.
Next, we immediately obtain
$$
\partial_{x_2}\btype(x) = C_p\,\partial_{y_2}\btype_0(r)\le 0.
$$

It remains to prove the first inequality in~\eqref{eq:diff_form}. As in~\cite{NaTr1996tran}, we obtain
\begin{equation*}
    d^2\btype[x](dx) 
    = C_p\,d^2\btype_0[r](d|x_1|, dx_2, dx_3, dx_4) + 
    C_p\,\partial_{y_1}\btype_0(r)\,d^2|x_1|.
\end{equation*} 
Here we mean that the differentials $d|x_1|$ and $d^2|x_1|$ are calculated at $x_1$ and are applied to $dx_1$: 
$$
    d|x_1| = dx_1\cdot e\quad\mbox{and}\quad
    d^2|x_1| = \frac{|Qdx_1|^2}{|x_1|},
$$
where $e \df \frac{x_1}{|x_1|}$ and $Qz \df z - (z\cdot e)\,e$.
Applying Lemma~\ref{le:B_0} and putting $C_p = \frac{1}{\sqrt{c_p}}$, we get
$$
    d^2\btype[x](dx) \le \frac{(d|x_1|)^2}{2\,\partial_{x_2}\btype(x)} + 
    \frac{2\,\partial_{y_1}\btype_0(r)\,\partial_{y_2}\btype_0(r)}{c_p|x_1|}   
    \frac{|Qdx_1|^2}{2\,\partial_{x_2}\btype(x)}.
$$
We have
$
    (dx_1\cdot e)^2 + |Qdx_1|^2 = |dx_1|^2.
$
By direct calculations, we get
$$
\partial_{y_1}\btype_0(r)\,\partial_{y_2}\btype_0(r)\ge \delta_p p\,|x_1|.
$$
Thus, any $\delta_p>0$ and $c_p \le 2\delta_p p$ give the desired result.
\end{proof}
Now we are ready to prove that $\btype\vert_{\Omega_p}\in\mathcal{K}^p\big(\seq{\ind}\big)$. 
In order to satisfy
the boundary condition~\ref{it:bnd_cnd} for 
$\btype$ on $\Omega_p$, we only need to take a sufficiently small~$\delta_p$ and to apply Young's inequality. 

Next, we prove that $\btype$ satisfies the concave-type condition~\ref{it:cnc_cnd} for all
$x$ and $x^{\pm}$ in $\seq{\ind}\times\Rpls^3$. 
We have
\begin{multline*}
\partial_{x_2}\btype(x_1,x_2) \df \partial_{x_2}\btype(x)\\=-
\begin{cases}
\gamma_1\,x_2^{q/2-1} + \gamma_2\,|x_1|^{2-p} + \gamma_3\,|x_1|^{2-p-2t_p(p-1)}\,x_2^{t_p}, &|x_1|^p \ge x_2^{q/2}; \\[2pt]
\gamma_4\,x_2^{q/2-1}, &|x_1|^p \le x_2^{q/2},
\end{cases}
\end{multline*}
where the constants $\gamma_i >0$ depend only on $p$. Whenever we deal with the function $\partial_{x_2}\btype$ below, we drop the variables $x_3$ and $x_4$ from the notation, because $\partial_{x_2}\btype$ does not depend on them.
We denote
\begin{gather*}
x^{\tau} = (x_1^{\tau},x_2^{\tau},x_3^{\tau},x_4^{\tau})\df \frac{1+\tau}{2}x^+ + \frac{1-\tau}{2}x^- ,
\\[2pt]
\phi(\rho) \df \btype\bigl(x^0 - (\bm 0,\rho\Delta^2,0,0)\bigr),
\quad\mbox{and}\quad
\Phi(\tau)\df \btype(x^{\tau}),
\end{gather*}
where $\tau \in [-1,1]$ and $\rho\in[0,1]$.
We have 
$$
\btype(x) - \frac{\btype(x^+)+\btype(x^-)}{2} = Q + R,
$$
where
$$
Q\df \phi(1)-\phi(0)\quad\mbox{and}\quad R\df \Phi(0) - \frac{\Phi(-1)+\Phi(1)}{2}.
$$

Calculating and reducing the interval of integration, we obtain 
$$
Q=\int\limits_0^1\phi'(\rho)\,d\rho 
\ge -\Delta^2\int\limits_0^{1/2} \partial_{x_2}\btype\bigl(x_1^{0},x_2^{0}-\rho\Delta^2\bigr)\, d\rho.
$$
Since $x_2^{0} \ge \Delta^2$, we have $x_2^{0}-\rho\Delta^2 \ge \frac{1}{2}x_2^{0}$ for $\rho \le \frac{1}{2}$. On the other hand, if there exists 
$\rho'\le\frac{1}{2}$ such that 
$\big|x_1^{0}\big|^p = \big(x_2^{0}-\rho'\Delta^2\big)^{q/2}$, then 
$\big|x_1^{0}\big|^p \asymp \big(x_2^{0}-\rho\Delta^2\big)^{q/2}$ for all $\rho \le \frac{1}{2}$.
In any case, we obtain 
$-\partial_{x_2}\btype\bigl(x_1^{0},x_2^{0}-\rho\Delta^2\bigr) \asymp -\partial_{x_2}\btype\bigl(x^{0}\bigr)$ for  
$\rho \le \frac{1}{2}$. Thus, we have
\begin{equation}\label{eq:Q}
Q\ge -c_p'\,\Delta^2\,\partial_{x_2}\btype\bigl(x^{0}\bigr).
\end{equation}

Now we consider the term~$R$. Without loss of generality, we may assume that 
$\ind$ is finite. 
Indeed, we can approximate $x_1$ and $x_1^\pm$ by sequences that contain only a finite number of non-zero elements. 
After that we can, due to the continuity of~$\btype$, pass to the limit 
in~\eqref{eq:main_ineq}. 
We can process situations where $x_1^{\tau}\equiv\bm{0}$ or where $x_2^{\tau}\equiv 0$ similarly: 
we can separate one of the points $x_1^\pm$ (or of the points~$x_2^\pm$, respectively) from zero and again pass to the limit 
in~\eqref{eq:main_ineq}.
In particular, these remarks allow us to regard the function~$\Phi'$ as absolutely continuous. 
Indeed, the continuous function~$\Phi'$ is differentiable on a \emph{co}finite set, and, as it can be seen below, $\Phi''(\tau) \le 0$ on this set. 
The former implies the Luzin N property for $\Phi'$, and the latter implies that $\Phi'$ is decreasing and, therefore, is of bounded variation. 
All this suffices for the absolute continuity of~$\Phi'$.

Next, we note that for any vectors $h$ and $b$ in $\seq{\ind}\times\mathbb{R}^3$, we have the following general relation. 
If we define, where possible, the function ${\Psi(\tau) \df \btype(\tau h + b)}$, then 
we have
\begin{equation}\label{eq:dbl_dir_dif}
\Psi''(\tau) = d^2\btype[\tau h+b] (h),
\end{equation}
where the right-hand side exists.
Further, we set $h \df \frac{x^+-x^-}{2}$.
Employing Taylor's formula in the integral form, relation~\eqref{eq:dbl_dir_dif}, and Lemma~\ref{le:B}, we obtain
\begin{align*}
R &= -\frac{1}{2}\int\limits_{-1}^{1}\Phi''(\tau)\,(1-|\tau|)\,d\tau =
-\frac{1}{2}\int\limits_{-1}^{1}d^2\btype[x^{\tau}](h)\,(1-|\tau|)\,d\tau 
\\&\ge - \frac{|h_1|^2}{4}\int\limits_{-1}^{1}\frac{1-|\tau|}{\partial_{x_2}\btype(x^{\tau})}\,d\tau.
\end{align*}
We have $|h_2|\le x_2^{0}$. First, we consider the case $|h_1|\le \big|x_1^{0}\big|$. 
For $\tau \in [-1/2,1/2]$, we have $x_2^{\tau} \asymp x_2^{0}$ and $\big|x_1^{\tau}\big| \asymp \big|x_1^{0}\big|$.
If there exists 
$\tau'\in[-1/2,1/2]$ such that $\big|x_1^{\tau'}\big|^p = \big(x_2^{\tau'}\big)^{q/2}$, then 
$\big|x_1^{\tau}\big|^p \asymp \big(x_2^{\tau}\big)^{q/2}$ for all $\tau\in[-1/2,1/2]$. All this implies
$$
R \ge -c_p''\frac{|h_1|^2}{4\,\partial_{x_2}\btype(x^{0})}.
$$
Combining this inequality with~\eqref{eq:Q}, we obtain the estimate 
\begin{equation}\label{eq:QR_est}
Q+R\ge \sqrt{c_p'\,c_p''}\,|h_1|\Delta.
\end{equation}
Next, suppose $|h_1|\ge \big|x_1^{0}\big|$. We have
\begin{equation}\label{eq:x_1}
    \big|x_1^{\tau}\big| \le 2|h_1| \quad\mbox{for}\quad \tau\in[-1,1].
\end{equation}
Let $S \subset [-1,1]$ be the set of all $\tau$ such that $x_2^{\tau}\ge x_2^{0} \ge \Delta^2$. 
We have $|S| = 1$ 
and $x_2^{\tau}\le 2 x_2^{0}$ for $\tau\in S$. If for all $\tau\in S$ we have 
$\big|x_1^{\tau}\big|^p \le \big(x_2^{\tau}\big)^{q/2}$, then we can estimate the integral over $S$ in the same way as for the case 
$|h_1|\le \big|x_1^{0}\big|$. Suppose there exists $\tau'\in S$ such that $\big|x_1^{\tau'}\big|^p > \big(x_2^{\tau'}\big)^{q/2}$. Then we have
$\big|x_1^{\tau'}\big|^p > \big(\frac{1}{2}x_2^{\tau}\big)^{q/2}$ for all $\tau \in S$. This implies
\begin{equation}\label{eq:x_2}
(2|h_1|)^p \ge 2^{-q/2}\big(x_2^{\tau}\big)^{q/2} \ge 2^{-q/2}\Delta^q\quad\mbox{for}\quad \tau\in S.
\end{equation}
Estimates~\eqref{eq:x_1} and~\eqref{eq:x_2} implies
\begin{equation}\label{eq:R_est}
R \ge c_p'''\, |h_1|\,|h_1|^{p-1} \ge c_p'''\,|h_1|\,\Delta.
\end{equation}
Since it is always true that \eqref{eq:QR_est} or~\eqref{eq:R_est} holds, we can adjust the constant~$C_p$ and obtain 
inequality~\eqref{eq:main_ineq} for~$\btype$.\qed 

\section{Proof of Corollary~\ref{cor:Lp_bnd}}
Suppose $g \in L^2$ and 
$$
x\df\Big(\av{f}{\ui},\,\|g\|\Cii{L^2}^2 - \osc\cii{\ui}^2(T^*g),\, 
\|f\|\Cii{L^p}^p,\,\|g\|\Cii{L^q}^q\Big).
$$
Let $\lambda > 0$. By the homogeneity of~$\bell$ and by Theorems~\ref{th:BleB} and~\ref{th:BinK}, we obtain
\begin{equation}\label{eq:cor_hom}
\begin{aligned}
    \av{g\,Tf}{\ui} - \av{f}{\ui}\av{T^*g}{\ui} &\le \bell(x_1,x_2,x_3,x_4)\\
    &= \bell\bigl(\lambda x_1,\lambda^{-2}x_2,\lambda^p x_3,\lambda^{-q}x_4\bigr)\\
    &\le \btype\bigl(\lambda x_1,\lambda^{-2}x_2,\lambda^p x_3,\lambda^{-q}x_4\bigr)\\
    &\le 2 C_p(\lambda^p x_3+\lambda^{-q}x_4).
\end{aligned}
\end{equation}
In order to guess optimal~$\lambda$, we need to solve the equation 
$\partial_{\lambda}\big[\lambda^p x_3 + \lambda^{-q}x_4\big] = 0$.
We obtain 
\begin{equation}\label{eq:cor_lmd}
\lambda = \biggl(\frac{q\,x_4}{p\,x_3}\biggr)^{\tfrac{1}{p+q}}.
\end{equation}
By Jensen's (or H\"older's) inequality, we have
\begin{equation}\label{eq:cor_Hld}
    \bigl|\av{f}{\ui}\av{T^*g}{\ui}\bigr| 
    \le \|f\|\Cii{L^p}\|T^*g\|\Cii{L^2}
    \le \|f\|\Cii{L^p}\|g\|\Cii{L^2}
    \le \|f\|\Cii{L^p}\|g\|\Cii{L^q}.
\end{equation}
Combining~\eqref{eq:cor_hom}, \eqref{eq:cor_lmd}, and~\eqref{eq:cor_Hld} for $g$ and $-g$, we 
obtain
$$
    |(g,Tf)| \le (2p^{1/p}q^{1/q}C_p + 1)\,\|f\|\Cii{L^p}\|g\|\Cii{L^q}.
$$
This finishes the proof.\qed

\section{Proof of Theorem~\ref{th:bell_is_cnc}} 
Jensen's (or H\"older's) inequality immediately implies that $\Omega_{\bell_0} \subseteq \Omega_p$. In order to prove 
that $\Omega_p\subseteq\Omega_{\bell_0}$, it suffices to set $T=\operatorname{id}_{L^2}$ and to choose, for $x \in \Omega_p$, functions~$f$ and~$g$ such that 
$\Av{f}{I} = x_1$, $\Av{g}{I} = \sqrt{x_2}$, $\Av{|f|^p}{I} = x_3$, and $\Av{|g|^q}{I} = x_4$. 
The desired functions can be easily found in the form
$$
f(t)=
\begin{cases}
    x_1 + a, & t \in I^+;\\
    x_1 - a, & t \in I^-,
\end{cases}
\quad
g(t)=
\begin{cases}
    \sqrt{x_2} + b, & t \in I^+;\\
    \sqrt{x_2} - b, & t \in I^-,
\end{cases}
$$
where $a,b\in\mathbb{R}$ are some chosen parameters.

Since the function~$|\cdot|^p$ is strictly convex, the case $|\Av{f}{I}|^p \le \Av{|f|^p}{I}$ of Jensen's inequality becomes the equality if and only if 
$f = \mathrm{const}$. Thus, we have~\ref{it:bnd_cnd} for~$\bell_0$.

It remains to prove~\ref{it:cnc_cnd}. Consider $x, x^\pm \in \Omega_p$ and $\Delta\in\mathbb{R}$ that are related with each other by~\eqref{eq:x_mean}.
For any $\varepsilon > 0$, there exist functions $f^{\pm},g^\pm\in L^2(I^\pm)$ and unitary operators 
$T^\pm\in\mathcal{G}(I^\pm,\mathbb{R})$ that generate $x^\pm$ and realize the supremum in~$\bell_0$ to an accuracy of~$\varepsilon$:
\begin{equation}\label{eq:pmsup}
\bigl\langle g^{\pm}\,T^{\pm}\bigl[f^{\pm}-\bigl\langle{f^{\pm}}\bigr\rangle_{I^{\pm}}\bigr]\bigr\rangle_{I^{\pm}} \ge \bell_0\bigl(x^\pm\bigr)-\varepsilon.
\end{equation}
Due to the unitarity of $T^{\pm}$, we have 
\begin{equation}\label{eq:x2Tpm}
    x_2^{\pm} = {|I^\pm|}^{-1}\bigl(g^{\pm},T^{\pm}h_0^{\pm}\bigr)^2,
\end{equation} 
where $h_0^{\pm} \df |I^{\pm}|^{-1/2}\chr_{I^{\pm}}$. In addition, we 
can always choose $T^{\pm}$ such that 
\begin{equation}\label{eq:x2sqrt}
    \bigl(g^{\pm},T^{\pm}h_0^{\pm}\bigr)\ge 0.
\end{equation}
We immediately see that the functions
$$
f(t)\df
\begin{cases}
    f^+(t), & t \in I^+;\\
    f^-(t), & t \in I^-
\end{cases}
\quad\mbox{and}\quad
g(t)\df
\begin{cases}
    g^+(t), & t \in I^+;\\
    g^-(t), & t \in I^-
\end{cases}  
$$
generates $x_1$, $x_3$ and $x_4$.
Now we construct an appropriate unitary operator ${T\in \mathcal{G}(I,\mathbb{R})}$. 
We set $Th_J \df T^\pm h_J$, $J \sqsubseteq I^\pm$, and build 
$Th_0$ and $Th_I$ in the form
\begin{align}
 Th_0 &= \xi_1T^+h_0^+ + \xi_2T^-h_0^-;\nonumber\\
 \label{eq:ThI}
 Th_I &= \xi_2T^+h_0^+ - \xi_1T^-h_0^-.
\end{align}
It is easy to see that $\{Th_0,Th_J\}\cii{{J \sqsubseteq I}}$ is an orthonormal basis in
$L^2(I)$ (and, therefore, $T$ is unitary), provided 
\begin{equation}\label{eq:xieq1}
\xi_1^2+\xi_2^2 = 1.
\end{equation}
If, in addition, we manage to choose $Th_I$ so that 
\begin{equation}\label{eq:dltThI}
\Delta = {|I|}^{-1}(g,Th_I),    
\end{equation}
then, first, relations~\eqref{eq:x_mean} and~\eqref{eq:osc_ser} will imply
$$
x_2 = \Av{g^2}{I}-\frac{\osc_{I^+}^2((T^+)^*[g^+])+\osc_{I^-}^2((T^-)^*[g^-])}{2}-\Delta^2
=\Av{g^2}{I} - \osc_{I}^2(T^*g)
$$
and, second, inequalities~\eqref{eq:pmsup} will imply
\begin{align*}
\bigl\langle g\,T\bigl[f-\Av{f}{I}\bigr]\bigr\rangle\Cii{I}
&=
\frac{1}{|I|}\Bigl((f,h_I)(g,Th_I) + \hspace{-10pt}\sum_{J\sqsubseteq I^+\!,\,J\sqsubseteq I^-}\hspace{-10pt} (f,h_J)(g,Th_J)\Bigr)\\[5pt]
&\ge 
\frac{x_1^+-x_1^-}{2}\,\Delta + \frac{\bell_0(x^+)+\bell_0(x^-)}{2} - \varepsilon.
\end{align*}
In such a case, $g$ and $T$ generate $x_2$ and, since $\Delta$ and $\varepsilon$ are arbitrary, we see that~\ref{it:cnc_cnd} holds for $\bell_0$.

It remains to prove that \eqref{eq:dltThI} is attainable. Substituting \eqref{eq:ThI} into \eqref{eq:dltThI} and using \eqref{eq:x2Tpm} and \eqref{eq:x2sqrt}, we come to the equation
\begin{equation}\label{eq:xieq2}
    \xi_2\sqrt{x_2^+}-\xi_1\sqrt{x_2^-} = \sqrt{2}\,\Delta.
\end{equation}
It is easy to calculate that the system of equations~\eqref{eq:xieq1} and~\eqref{eq:xieq2} is solvable exactly
when
$$
    \frac{x_2^+ + x_2^-}{2} - \Delta^2 \ge 0.
$$
This is true due to \eqref{eq:x_mean}, and we are done. 
\qed

\section{Acknowledgments}
The authors are grateful to V.~I.~Vasyunin and D.~M.~Stolyarov for valuable comments, which helped them to choose the right line of reasoning. 
The authors also thank S.~V.~Kislyakov for pointing out that it may be possible to apply the Bellman function method to the discrete variant of the general
Calder\'on--Zygmund theory. 
The second author is also grateful to A.~L.~Volberg for the fruitful discussion during his visit to MSU.

\printbibliography
\end{document}